\documentclass[12pt]{amsart}
\usepackage[colorlinks, backref]{hyperref}
\usepackage[T1]{fontenc}
\usepackage{amscd}
\usepackage{amssymb}
\usepackage[leqno]{amsmath}
\usepackage{amsfonts}
\usepackage{latexsym}
\usepackage{amsopn}
\usepackage{amstext}
\usepackage{amsthm}

\usepackage[all]{xy}
\let\SavedRightarrow=\Rightarrow
\usepackage{marvosym}
\let\Rightarrow=\SavedRightarrow
\newenvironment{pf}{\begin{proof}}{\end{proof}}

\newcommand{\RO}{\operatorname{RO}}

\newcommand{\CO}{\operatorname{Clop}}

\newcommand{\coZ}{\operatorname{coZ}}

\newcommand{\Aaa }{\mathcal A}
\newcommand{\Raa }{\mathcal R}
\newcommand{\Bee }{\mathcal B}
\newcommand{\Dee }{\mathcal D}
\newcommand{\Cee }{\mathcal C}
\newcommand{\Wee }{\mathcal W}

\newcommand{\Pee }{\mathcal P}

\newcommand{\Ess }{\mathcal S}

\newcommand{\w}{\operatorname{w}}
\newcommand{\low}{\operatorname{l}}
\newcommand{\up}{\operatorname{u}}

\newcommand{\cl}{\operatorname{cl}}
\renewcommand{\int}{\operatorname{int}}

\newcommand{\s}{\operatorname{s}}
\newcommand{\p}{\operatorname{p}}
\newcommand{\Ss}{\operatorname{\overline{s}}}

\newcommand{\liminv}{\varprojlim}
\newcommand{\FNS}{\operatorname{FNS}}
\newcommand{\FN}{\operatorname{FN}}
\newcommand{\St}{\operatorname{St}}
\newcommand{\Ult}{\operatorname{Ult}}

\newtheorem{theorem}{Theorem}
\newtheorem*{tw}{Theorem}
\newtheorem{corollary}{Corollary}

\newtheorem{lemma}{Lemma}
\newtheorem{question}{Question}
\newtheorem{proposition}[theorem]{Proposition}

\theoremstyle{definition}
\newtheorem{example}{Example}

\author{Judyta B\k{a}k} 
\address{Judyta B\k{a}k\\  Institute of Mathematics, Jan Kochanowski University, 
\'{S}wi\k{e}tokrzyska~15,
25-406 Kielce and Institute of Mathematics, University of Silesia \\  Bankowa 14, 40-007 Katowice\\ Poland} \email{jubak@us.edu.pl}

\author{Andrzej Kucharski} 
\address{Andrzej Kucharski \\  Institute of Mathematics, University of Silesia in Katowice\\  Bankowa 14, 40-007 Katowice\\ Poland} \email{andrzej.kucharski@us.edu.pl} 
\RequirePackage[normalem]{ulem} 
\RequirePackage{color}\definecolor{RED}{rgb}{1,0,0}\definecolor{BLUE}{rgb}{0,0,1} 

\begin{document}

\begin{abstract}
The aim of this paper is to study the class of
 spaces with the FNS property and $\pi-\FNS$ property. We shown that compact spaces with the FNS property for some base consisting of cozero-sets are openly generated spaces and spaces with the $\pi-\FNS$ property are skeletally generated spaces.
\end{abstract}

\title{Topological spaces with the Freese--Nation property II} 

\maketitle
 
\section{Introduction}
This paper is a continuation of  the Freese-Nation property  research introduced in the  paper \cite{bk} for topological spaces. The Freese-Nation property was introduced by  R. Freese and J.B. Nation \cite{FN}. In \cite{bk} it is  shown that  spaces with the FNS property satisfy ccc and any product of such spaces also satisfies ccc. All metrizable spaces have the FN property.   L. Heindorf and L.B. Shapiro \cite{npba}  showed that a family of all clopen sets of  0-dimensional compact space $X$  has the FNS property if and only if $X$ is openly generated.  The concept of openly generated spaces was introduced by E.V. Shchepin in \cite{s76} and developed  in \cite{s79} and \cite{s81}. The FNS property and some versions of it for compact spaces were studied in  \cite{npba}, \cite{FKS}, \cite{FS}, \cite{Mi18}, \cite{Mi12}.

We say that a family $\Bee$ of open sets has \textit{$\FNS$ property} if there exists a map $\s:\Bee\to[\Bee]^{<\omega}$ such that if $U,V\in\Bee$ are disjoint then there are disjoint sets  $W_U, W_V\in s(U)
\cap s(V)$ such that $U\subseteq W_U, V\subset W_V$.

We say that a topological space $X$ has \textit{$\FNS$ property} if there exists a base which has $\FNS$ property. We say that a space $X$ has \textit{$\pi-\FNS$ property} if there exists a $\pi$-base which has $\FNS$ property.

A topological space has the $\FN$ \textit{property} if there exists a base $\mathcal{B}$  such that for every $V\in \mathcal{B}$ there are two finite sets $\up(V) \subseteq \{U\in \mathcal{B}:V\subseteq U\}$ and $\low(V) \subseteq \{U\in \mathcal{B}:U\subseteq V\}$ such that if $V\subseteq W$, then $\up(V)\cap \low(W)\ne \emptyset$.

A regular space with countable weight has the  $\FNS$ property and a regular space with countable $\pi$-weight has the  $\pi-\FNS$ property. We estabilish in Section 2 that every compact Hausdorff  space  with the $\FNS$ property for some base consisting of cozero-sets is openly generated. We prove that Player I in open-open game has winning strategy in every space with the  $\pi-\FNS$ property. It  follows from the result \cite{kp8} that every compact Hausdorff  space  with the $\pi-\FNS$ property is skeletally generated. In Section 3 we show that developable spaces have the FN property. We indicate examples of base without $\FNS$ property. Section 4 contains the proof that a   space $Y$ which is co-absolute to a  space $X$ with the $\pi-\FNS$ property  has the $\pi-\FNS$ property.

\section{Openly generated spaced and skeletally generared spaces}

We say that a compact  space $X$ is \textit{openly generated} if it is homeomorphic to $\liminv\{X_\sigma,p^{\sigma '}_\sigma,\Sigma\}$, where
 \begin{enumerate}
\item[(1)] $X_\sigma$ is a compact metrizable space for $\sigma \in \Sigma$,
\item[(2)]$p_\sigma^{\sigma '}\colon X_{\sigma '}\to X_\sigma$ is an open surjection,
\item[(3)] $\Sigma$ is $\sigma$--complete, i.e. for any chain $\{\sigma_n:n\in \omega\}\subseteq \Sigma$ there exists $\sigma=\sup\{\sigma_n, n\in \omega\}\in \Sigma$,
\item[(4)]the inverse system is continuous, i.e. $X_\sigma=\liminv\{X_{\sigma_n},p^{\sigma_{n+1}}_{\sigma_n},\omega\}$.
\end{enumerate}

We say that a map $f\colon X \to Y$ is \textit{skeletal} \cite{mr} if $\int \cl f(U)\ne \emptyset$ for every nonempty open set $U\subseteq X$. 

Similarly as openly generated space we define \textit{skeletally generated space}, by replacing open surjections $p_\sigma^{\sigma '}\colon X_{\sigma '}\to X_\sigma$ with skeletal surjections. Skeletally generated spaces were introduced in \cite{va}.

Let $X$ and $Y$ be a topological spaces. We say that a function $f\colon X\to Y$ is \textit{d-open} \cite{tkachenko} if $f(U)\subseteq \int \cl f(U)$ for every open set $U\subseteq X$.

\begin{lemma}[\cite{tkachenko}]\label{d-open} 
Let $f\colon X\to Y$ and $f(U)\subseteq \int \cl f(U)$ for every set $U\in \mathcal{B}$, where $\mathcal{B}$ is a base of regular space $X$, a map $f$ is closed, then $f$ is open.
\end{lemma}

\begin{lemma}[\cite{kpv}]\label{KPV} For any map $f\colon X \to Y$ the following conditions are equivalent:
\begin{itemize}
\item[(1)] f is d-open,
\item[(2)]\label{2} there exists a base $\mathcal{B}$ of a space $Y$ such that a family $\mathcal{P}=\{f^{-1}(V ) : V \in \mathcal{B}\}$ satisfies the following condition: for every family $\mathcal{S}\subseteq \mathcal{P}$ and a point $x\notin\cl\bigcup \mathcal{S}$, there is an open neighborhood $W\in \mathcal{P}$ of $x$ such that $W\cap \bigcup \mathcal{S}=\emptyset$,

\end{itemize}
\end{lemma}

Let $\mathcal{P}$ be a family of open subsets of a space $X$. We introduce an equivalence relation $\sim_\mathcal{P}$ on a space $X$ in the following way
$$x\sim_\mathcal{P} y \quad \Longleftrightarrow \quad (\forall_{V\in \mathcal{P}} \quad x\in V \Leftrightarrow y\in V).$$
Let $X/ \mathcal{P}$ be a set of all equivalence classes with the quotient topology and let $q_\mathcal{P}\colon X\to X/\mathcal{P}$ denote a quotient map (see \cite{eng}).

Following \cite{ker} (see also \cite{bran-mys}) let us call a family $\Pee$ of open subsets of a spaces $X$ \textit{completely regular }if for each $U\in\Pee$ there exist two  sequences $\{V_n: n\in\omega\}$ and $\{U_n: n\in\omega\}$ in $\Pee$ such that
\begin{equation*}
U=\bigcup_{n\in\omega} U_n \text{ and }U_n\subseteq X\setminus V_n\subseteq U\text{ for each }n\in\omega.
\end{equation*}

J. Kerstan has proved the following characterization of complete regularity.

\begin{tw}[\cite{ker},(see also \cite{bran-mys})]\label{ker}
	A $T_1$-space is completely regular if and only if it has a completely regular base.
\end{tw}

We say that a family $\Pee$ of open subsets of a spaces $X$ \textit{weakly completely regular } if for each ${k\in \omega}$ and any  $U_1,\ldots , U_k\in\Pee$ there exist two  sequences $\{\mathcal{A}_n:n\in \omega\},$ $\{\mathcal{B}_n:n\in \omega\}\subseteq [\mathcal{P}]^{<\omega} $ such that
\begin{multline*}
U_1\cap \ldots \cap U_k=\bigcup \{ \bigcup \mathcal{A}_n: n\in \omega\}= \bigcup \{X \setminus \bigcup \mathcal{B}_n: n\in \omega\} \\\text { and } \bigcup \mathcal{A}_n\subseteq X \setminus \bigcup \mathcal{B}_n \text{ for each } n\in \omega \tag{$*$}
\end{multline*} 
 
 If a  space $X$ has a completely regular family $\Pee$ then not so hard to see that $\Pee'$ closed under the finite intersection and unions of elements of $\Pee$ is  completely regular family. Then $\Pee'$ is weakly completely regular family. Next theorem is similar to Theorem 2.16. \cite{kt19} but we assume only that family $\Pee$ is a countable and weakly completely regular.

\begin{theorem}\label{metr}
	If 	$\Pee$ is a countable and weakly completely regular family of  $X,$ then $X_\Pee$ is metrizable.
\end{theorem}
\begin{proof}
We will show that a  map $q_\mathcal{P}\colon X\to X/\mathcal{P}$ is continuous. First we check that $\{q_\mathcal{P}[V]:V\in 
\mathcal{P}\}$ is a base of the space $X/\mathcal{P}$. To show that $\bigcup \mathcal{P}=X,$ take any $x\in X$ and $W\in \mathcal{P}$. We 
are not losing generality assuming that $x\notin W.$ Since $\Pee$ is weakly completely regular there is a sequence $\{\mathcal{B}_n:n\in \omega\}\subseteq [\mathcal{P}]^{<\omega} $ such that $W=\bigcup \{X \setminus \bigcup 
\mathcal{B}_n: n\in \omega\}.$ Hence we get  $x\in \bigcup 
\mathcal{B}_0\subseteq \bigcup \mathcal{P}$. This proves that  $\bigcup \{q[V]:V\in \mathcal{P}\}=X/\mathcal{P}$.
Now we shall prove that 
$$q[V\cap W]=q[V]\cap q[W]\text{ for all }U,W\in \mathcal{P}.$$
 Take any $V, W\in\Pee.$ Clearly $q[V\cap W]\subseteq q[V]\cap q[W]$. Take an arbitrary  $q[x]\in q[V]\cap 
q[W]$, then there is $y\in V$ such that $q[x]=q[y]$. By the definition of the relation $\sim_\mathcal{P}$ we get  $x\in V.$  The same reasoning shows that $x\in W$ so we proved that $q[x]\in q[V\cap W]$. Since $\Pee$ is the weakly completely regular family there is a sequence $\{\mathcal{A}_n:n\in \omega\}\subseteq [\mathcal{P}]^{<\omega} $ such that $V\cap W=\bigcup\{ \bigcup \mathcal{A}_n:n\in \omega\}.$  Therefore  there exists $U\in \bigcup \{\mathcal{A}_n:n\in \omega\}$ such that $q[x]\in q[U] \subseteq q[V\cap W]=q[V]\cap q[W]$. This proves that  the family $\{q[V]:V\in 
\mathcal{P}\}$ is the base of the space $X/\mathcal{P}$. 

	The map $q$ is continuous, because $q^{-1}(q[V])=V$ for all $V\in \mathcal{P}$.
To show that the space $X/\mathcal{P}$ is Hausdorff space, take $q[x]\ne q[y].$ There exists $W\in 
\mathcal{P}$ such that $x\in W$ and $y\notin W$. Since $\Pee$ is the weakly completely regular family there are  sequences $\{\mathcal{A}_n:n\in \omega\},\{\mathcal{B}_n:n\in \omega\}\subseteq [\mathcal{P}]^{<\omega}$ which satisfy condition $(*)$ for the set $W.$ There is $n\in \omega$ such that $x\in \bigcup \mathcal{A}_n \subseteq X\setminus 
\bigcup \mathcal{B}_n\subseteq W$ and $y\notin X\setminus  \bigcup \mathcal{B}_n$. Therefore there are neighborhoods $U_x\in 
\mathcal{A}_n\subseteq \mathcal{P}$ of the point $x$ and $U_y\in \mathcal{B}_n\subseteq \mathcal{P}$ of the point $y$. Since $\bigcup 
\mathcal{A}_n \cap \bigcup \mathcal{B}_n = \emptyset$ we have  $q[U_x]\cap q[U_y]=q[U_x\cap U_y]=\emptyset$ and 
$q[x]\in q[U_x],$ $q[y]\in q[U_y]$. 

We shall prove that $X/\mathcal{P}$ is regular. Take any closed set $F\subseteq X/\mathcal{P}$ and a point $q(x)\not\in F$. Since $\{q_\mathcal{P}[V]:V\in \mathcal{P}\}$ is a base of the space $X/\mathcal{P}$ there is $V\in\Pee$ such that $q(x)\in q[V]$ and $q[V]\cap F=\emptyset.$ By the property $(*)$ of $\Pee$ there are sequences $\{\mathcal{A}_n:n\in \omega\},\{\mathcal{B}_n:n\in \omega\}\subseteq [\mathcal{P}]^{<\omega}$ which satisfy condition $(*)$ for the set $V.$ So there is $n\in\omega$ such that $q(x)\in q[U]\in\{q[G]:G\in\mathcal{A}_n\}$ and $X/\mathcal{P}\setminus \bigcup\{q[G]:G\in\mathcal B_n\}\subseteq X/\mathcal{P}\setminus F.$
Since $\bigcup \mathcal{A}_n\subseteq X \setminus \bigcup \mathcal{B}_n$
the space $X/\mathcal{P}$ is regular.	
By the Urysohn metrization theorem $X/\mathcal{P}$  is metrizable. 
\end{proof}

We get a similar characterization of complete regularity by the property $(*)$ obtained by J. Kerstan \cite{ker} ( see also \cite{bran-mys}).

\begin{corollary}
	A $T_0$-space is completely regular if and only if it has a weakly completely regular base.
\end{corollary} 
\begin{proof}Clearly the family of all cozero-set of a space is weakly comletely regular. 
Let $\Bee$ be a weakly completely regular base for a space $X$. To show that $X$ is completely regular take any $U\in\Bee$. One can assume that there is  a countable subfamily  $\Pee\subseteq \Bee$ such that $\Pee$ is  weakly completely regular and $U\in\Pee$. By Theorem \ref{metr} $X/\mathcal{P}$  is metrizable. Since $q^{-1}(q[V])=V$ for all  $V\in \mathcal{P}$ and $q[U]$ is an open set in the metrizable space $X/\mathcal{P}$,   $U$ is a cozero-set, this completes the proof.
\end{proof}
\begin{theorem}\label{pog}
Every compact Hausdorff  space  with the $\FNS$ property for some base consisting of cozero-sets is openly generated.
\end{theorem}

\begin{proof}
Let $\mathcal{B}$ be a base consisting of cozero-sets of a space $X$ which has the $\FNS$ property. For every countable family $\mathcal{A}\subseteq \mathcal{B}$ there exists a countable weakly completely regular family $\mathcal{P}$ such that $\mathcal{A}\subseteq \mathcal{P} \subseteq \mathcal{B},$ and  $s(V)\subseteq \mathcal{P}$ for all $V\in\Pee.$  
Indeed, if $U_1,\ldots ,U_k\in \coZ(X)$, then $W=U_1\cap \ldots \cap U_k\in \coZ(X)$ and $W$ is $F_\sigma$ set. Therefore there exists a continuous function $f\colon X\to [0,1]$ such that
$$W=f^{-1}((0,1])=\bigcup\{F_n:n\in \omega\},$$
$$\text{where }F_n=f^{-1}([\frac{1}{n}, 1])\subseteq f^{-1}((\frac{1}{n+1}, 1])\subseteq f^{-1}([\frac{1}{n+1}, 1]).$$

Let $U_n=f^{-1}((\frac{1}{n}, 1])$, then $F_n\subseteq U_{n+1}$ for $n\in \omega$. Since the space $X$ is regular, for every $x\in F_n$ there exists $V_x\in \mathcal{B}$ such that $x\in V_x\subseteq \cl V_x \subseteq U_{n+1}$. Set $F_n$ is compact for $n\in \omega$, hence $F_n\subseteq \bigcup\{V_{x_i}:i\le k\}\subseteq \bigcup \{\cl V_{x_i}:i\le k\} =\cl (\bigcup \{V_{x_i}:i\le k\})\subseteq U_{n+1}$. Let $\mathcal{A}_n=\{V_{x_i}:i\le k\}.$ Similarly for $X\setminus U_{n+1}\subseteq X\setminus \cl \bigcup \mathcal{A}_n$ we get $\mathcal{B}_n\in [\mathcal{B}]^{<\omega}$ such that $X\setminus U_{n+1}\subseteq \bigcup \mathcal{B}_n \subseteq X\setminus \cl \bigcup \mathcal{A}_n$, then $\cl \bigcup \mathcal{A}_n \subseteq X\setminus \bigcup \mathcal{B}_n\subseteq U_{n+1}$. For every $n\in \omega$ we have $F_n\subseteq \bigcup \mathcal{A}_n \subseteq \cl \bigcup \mathcal{A}_n\subseteq  X\setminus \bigcup \mathcal{B}_n\subseteq U_{n+1}$, hence $W=\bigcup\{F_n:n\in \omega\}=\bigcup\{ \bigcup \mathcal{A}_n:n\in \omega\} =\bigcup \{  X\setminus \bigcup \mathcal{B}_n:n\in \omega\}$.

By Theorem \ref{metr} the space $X/\mathcal{P}$ is metrizable.
We shall show that $q\colon X\to X/\mathcal{P}$ is an open map. The map $q\colon X\to X/\mathcal{P}$ is closed map as a continuous map from a compact space to a Hausdorff space. According to the lemma \ref{d-open} it is sufficient to show that a map $q$ is d-open. By the lemma \ref{KPV} a map $q\colon X\to X/ \mathcal{P}$  is d-open if and only if a family $\mathcal{P}$ satisfies the condition \ref{2} of this lemma. Take a family $\mathcal{S} \subseteq \mathcal{P}$ and $x\notin \cl \bigcup \mathcal{S}.$ There exists $V\in \mathcal{B}$ such that $x\in V$ and $V\cap \bigcup \mathcal{S}=\emptyset$. Put $\Raa=\{V'':V\subseteq V'', V''\in s(V)\cap \mathcal{P}\}.$ Since  the base $\mathcal{B}$ has the FNS property  there are sets $V',U'\in s(V) \cap s(U)\subseteq \mathcal{P}$ such that $V\subseteq V'$, $U\subseteq U'$ and $V'\cap U'=\emptyset$ for every $U\in \mathcal{S}$. Therefore $\bigcap\Raa\cap  \bigcup \mathcal{S}=\emptyset$.
Since $\Pee$ is the weakly completely regular family there is a sequence $\{ \bigcup \mathcal{A}_n: n\in \omega\}\subseteq[\Pee]^{<\omega}$ such that  $\bigcap\Raa=\bigcup\{\bigcup\mathcal{A}_n:
n\in\omega\}.$ Since $x\in V\subseteq \bigcap\Raa$ there is $m\in \omega$ and  $W\in \mathcal{A}_m$ such that $x\in W\subseteq \bigcap\Raa$. Hence $W\cap \bigcup \mathcal{S}=\emptyset$ this proves that the map $q$ is d-open. 

Put
\begin{multline*}
\Sigma=\{\mathcal{P}\in[\mathcal{B}]^{<\omega} :\Pee\text{ is a  weakly completely regular family }\\
\text{ and }s(V)\subseteq \mathcal{P} \text{ for all } V\in\Pee\}.
\end{multline*}
The family $\Sigma$ ordered by inclusion is directed by the first part of the proof. If $\Pee\in\Sigma$ then a map $q:X\to X/\mathcal{P}$ is an open surjection onto the compact metrizable space $X/\mathcal{P}.$ If $\Pee\subseteq\Raa$, where $\Pee,\Raa\in\Sigma$, then one can define a map $\pi_{\Pee}^{\Raa}\colon X_{\Raa}\to X_{\Pee}$ such that $\pi_{\Pee}^{\Raa}\circ q_{\Raa}=q_{\Pee}$. Clearly $\pi_\Pee^\Raa$ is open. If $\{\Pee_n\colon n\in\omega\}$ is an increasing chain  in $\Sigma$ then the space $X_{\Pee}$,  is homeomorphic to $\varprojlim\{X_{\Pee_n},\omega \},$ where $\Pee=\bigcup\{\Pee_n\colon n\in\omega\}$.	Thus
$\{X_\Pee,\pi_\Pee^\Raa,\Sigma\}$ is a $\sigma$-complete inverse system,  all spaces $X_\Pee$ are compact and metrizable and all bonding maps $\pi_\Pee^\Raa$ are open. A map $h\colon X \to \liminv\{X/\mathcal{P},q^{\mathcal{P}_2}_{\mathcal{P}_1},\Sigma\}$ given by the formula
$$h(x)=\{[x]_{\mathcal{P}}\}_{\mathcal{P}\in \Sigma}$$
is the homeomorphism.
\end{proof}

\begin{theorem}{\label{ro}}
A family of all regular open sets in regular infinite space does not have the $\FNS$ property. 
\end{theorem}

\begin{proof}
Suppose that a regular infinite space $X$  has  a map $\s\colon \RO (X)\to [\RO (X)]^{<\omega}$ witnessing the $\FNS$ property for $\RO(X).$ Consider the Stone  space $\Ult(\RO (X))$ of the Boole'a algebra $\RO (X)$ with a topology generated by a base
$$\mathcal{B}=\{\overline{U}: U\in \RO (X)\},$$
where
$$\overline{U}=\{\mathcal{F} \in \Ult (\RO (X)):  U\in \mathcal{F} \}.$$
 Let  $\Ss \colon \mathcal{B}\to [\mathcal{B}]^{<\omega}$ be given by the formula
$$\Ss (\overline{U})=\{\overline{W}: W\in \s(U)\}.$$
Since  $\overline{U}\cap \overline{V}=\overline{U\cap V}$ for all $U, V\in \RO (X) ,$  the family $\mathcal{B}$ has the $\FNS$ property.

$\Ult(\RO (X))$ is the compact Hausdorff space with the $\FNS$ property. By  Theorem \ref{pog} $\Ult(\RO (X))$ is openly generated and $\Ult(\RO (X))=\liminv\{X_\sigma,\p^\sigma_\rho,\Sigma\},$ where $\Ess=\{X_\sigma,\p_\sigma^\rho,\Sigma\}$ is a $\sigma$-complete inverse system,  all spaces $X_\sigma$ are compact and metrizable and all bonding maps $\p_\sigma^\rho$ are open.   Since the space  $\Ult(\RO (X))$ is extremally disconnected and each $\p_\sigma$ is an open map,  every $X_\sigma$ is extremally disconnected. The space $X_\sigma$ is extremally disconnected and compact metric, therefore it has to be finite. 

Since  $X$ is the infinite regular space, $|\RO(X)|>\omega.$ For each $n\in \omega$ there exists $\sigma_n\in \Sigma$ such that  $\w(X_{\sigma_n})>n$ and $ \{\sigma_n : n\in \omega\}$ is a chain. Since $\Ess$ is the $\sigma$-complete inverse system $\w(X_\sigma)\geq \omega$ where $\sigma=\sup\{\sigma_n : n\in \omega\},$ a contradiction. 
\end{proof}

\begin{corollary}\label{top-not-FNS}
A topology of regular infinite space does not have the $\FNS$ property.
\end{corollary}

\begin{pf}
Assume that the topology of regular infinite space $X$ has the $\FNS$ property. From the Proposition 2.2 \cite{bk} it follows that the family of all regular open sets $\RO(X)$ has the $\FNS$ property, a contardiction with Proposition \ref{ro}.   
\end{pf} 

From the above proposition follows that $\beta \mathbb{N}$ has the $\pi- \FNS$ property and does not have the $\FNS$ property.

\begin{example} Suppose that there exists a base $\Bee$ of $\beta \mathbb{N}$ with the $\FNS$ property. Then $\Bee'=\{\cl U:U\in\Bee\}$ is a base consisting of  clopen sets. It's easy to see that $\Bee'$ has the $\FNS$ property. By Theorem \ref{pog} there exists a $\sigma$-complete inverse system $\Ess=\{X_\sigma,\p_\sigma^\rho,\Sigma\}$ such that  all spaces $X_\sigma$ are compact and metrizable and all maps $\p_\sigma^\rho,\p_\sigma$ are open and $\liminv \Ess=\beta \mathbb{N},$ a contradiction because  $\beta \mathbb{N}$ is extremally disconnected (compare the proof of Theorem \ref{ro}). Since 
$\beta \mathbb{N}$ has a countable $\pi-$ base $\Bee$ closed under the complement it easy to define operator $\s:\Bee\to[\Bee]^{<\omega}$ which is witness on the $\FNS$ property (compare \cite[Proposition 5]{bk}).
\end{example}

\begin{corollary}
	If $X$ is a regular infinite space, then the base $\RO(X)$ does not have the $\FNS$ property.
\end{corollary}

Now we shall prove that if $X$ is a regular infinite space, then the bases $\RO(X)$ and the topology of $X$ do not have the FN property.

Consider a cardinal number $\kappa>\omega$ with the discrete topology. It easy to see that a base 
$\{\{\alpha\}:\alpha<\kappa\}$ has the $\FN$ property. One can check that 
$\up(\{\alpha\})=\{\{\alpha\}\}=\low(\{\alpha\})$ 
for all $\alpha<\kappa$ have the desired properties. Next theorem gvies a negative answer to the question: Does the base $\RO(X)$ ( topology) have 
the $\FN$ property, whenever there exists some base with $\FN$ property?

\begin{theorem}\label{no-FN}
A family of all regular open sets in regular infinite space does not have the $\FN$ property.
\end{theorem}
\begin{pf}
Let $X$ be a regular infinite space. Suppose that  operators  
$\up,\low $ are witnesses on the $\FN$ property for a family of all regular open sets $\RO(X).$ 
Since the space $X$ is regular and infinite 
there is an infinite maximal family $\Raa\subseteq\RO(X)$ of pairwise disjoint set. 
 We may define by a straight forward recursion a sequence 
$\{V_n:n\in\omega\}\subseteq\RO(X)$ and a sequence of  
finite sets $\{\Pee_n\in[\Raa]^{<\omega}:n\in\omega\}$ such that 
\begin{enumerate}
	\item $V_n\subsetneq V_{n+1}$ and $\Pee_n\subsetneq\Pee_{n+1}$ for all $n\in\omega,$
	\item $V_n\subseteq \bigcup(\Raa\setminus\Pee_n)$  for all $n\in\omega,$
	\item $V_n\cap W=\emptyset$ or $W\subseteq V_n$ for all $W\in\Raa$ and $n\in\omega,$
	\item $\{W\in\Raa:W\subseteq V_n\}$ is finite for all $n\in\omega,$
	\item If $V_n\subseteq W\subseteq\bigcup(\Raa\setminus\Pee_n),$ 
	then $V_n\in\low(W)$, for each $W\in\Bee$ and $n\in\omega$.
\end{enumerate}
Suppose that we have just defined $\Pee_i\in [\Raa]^{<\omega}$ and $V_i\in\RO(X)$ with the property $(1)-(5)$ 
for $ i\leq n$. For each $W\in\up(V_n)\setminus\{V_n\}$ there is $U\in\Raa$ such that $W\cap U\ne\emptyset$ and
$U\cap V_n=\emptyset.$ Therefore there is a finite family $\Pee_{n+1}\subseteq \Raa$ such that $\Pee_n\subseteq\Pee_{n+1}$ and
$\bigcup\Pee_{n+1}\cap W\ne\emptyset$, whenever $W\in\up(V_n)\setminus\{V_n\}.$ Finaly we get $\{V_n:n\in\omega\}\subseteq\RO(X)$ and $\{\Pee_n\in[\Raa]^{<\omega}:n\in\omega\}$  which satisfy properties $(1)-(5).$ Let $\Pee=\bigcup\{\Pee_n:n\in\omega\}$ and $U=\int\cl\bigcup\{V_n:n\in\omega\}.$ For every $n\in\omega$ we get  $$V_n\subseteq \bigcup\{V_n:n\in\omega\}\subseteq U\subseteq \bigcup(\Raa\setminus\Pee)\subseteq \bigcup(\Raa\setminus\Pee_n).$$
 Hence  $V_n\in\low(U)$ for every $n\in\omega$ by property $(5)$, a contradiction.
\end{pf}
\begin{corollary}\label{top-not-FN}
	A topology of regular infinite space does not have the $\FN$ property.
\end{corollary}

\begin{pf}
	Assume that the topology of regular infinite space $X$ has the $\FN$ property. Suppose that   operators  
	$\up,\low $ are witnesses on the $\FN$ property for a family of all open sets. Put $$\up_{reg}(U)=\{\int\cl W: W\in\up(U)\}
	\text{ and }\low_{reg}(U)=\{\int\cl W: W\in\low(U)\},$$
	for every $U\in\RO(X).$  The operators  
	$\up_{reg},\low_{reg} $ are witnesses on the $\FN$ property for a family of all regular open sets, a contradition with 
	Theorem \ref{top-not-FN}.   
\end{pf} 

In the paper \cite{bk} we have proved that if there exists a base $\Bee$ closed under the finite intersections with 
the FNS property, then one can enlarge the base $\Bee$ to a base $\Bee'$ with FNS that contains both families 
$\Bee$ and $\CO(X)$.
Now we can strengthen the mentioned result.

\begin{theorem}
	If  a 0-dimensional compact space $X$ has the $\FNS$ property for some base consisting of cozero-sets, then  
	$\CO(X)$ has the $\FNS$ property.
\end{theorem}
\begin{pf}
By Theorem \ref{pog} $X$ is openly genereted space, i.e. $\Ess=\{X_\Pee,\pi_\Pee^\Raa,\Sigma\}$ is a $\sigma$-complete
 inverse system,  all spaces $X_\Pee$ are compact and metrizable and all bonding maps $\pi_\Pee^\Raa$ are open and
  $X=\liminv\Ess$. We may assume without loss of generality that	$X_\Pee$ are 0-dimensional. Applying Heindorf 
  and Shapiro's result \cite[Theorem 2.2.3]{npba} ( see also \cite[Theorem 21]{bk}) we get an operator which witnesses 
  the FNS property for the family $\CO(X)$.
\end{pf}
The open-open game was introduced by P. Daniels, K. Kunen  and H.~Zhou \cite{dkz} for an arbitrary topological space $X$.
 Two players take countably many turns. A round consists of player I choosing a nonempty set $U$ and player II choosing a
  nonempty $V$ with
$V\subseteq U$.
Player I wins if the union of II's open sets is dense in $X$, otherwise player II wins. Denote this game by $G(X)$.  
Consider the following game $G_\omega(X)$ on a topological space $X$. At the $n$-th round Player I chooses a  finite 
family $\Cee_n$ of  open non-empty open sets. Then Player II chooses a finite family $\Dee_n$ of non-empty open  
subsets of $X$ such that for each $U \in \Cee_n$ there exists $V \in \Dee_n$ with $V \subseteq U $. 
Player I wins if the union of II's open sets is dense in $X$, otherwise player II wins. It's well known that  
the game $G(X)$ is equvalent to the game $G_\omega(X)$ (see \cite{dkz}). We say that Player I has a winning strategy 
in the game $G_\omega(X)$ whenever  there exists a function  
$$( \Dee_0,\Dee_1 , \ldots, \Dee_n) \mapsto \sigma ( \Dee_0,\Dee_1 , \ldots, \Dee_n), $$ 
where all families $\Dee_n$ and $\sigma ( \Dee_0,\Dee_1 , \ldots, \Dee_n)$ are finite and  
consists of non-empty open sets such that for each game
$$\sigma(\varnothing),\Dee_0,\sigma(\Dee_0),\Dee_1,\sigma(\Dee_0,\Dee_1),\Dee_2,\ldots,
\Dee_n,\sigma(\Dee_0,\ldots,\Dee_n),\Dee_{n+1},\ldots$$
the union  $\bigcup_{n\geq 0}\bigcup\Dee_n$ is dense in $X$. Assume that $\Bee$ is an $\pi-$ base, 
then  I Player has a winning strategy in the game $G(X)$ if and only if there is a winning strategy  
defined on $\pi-$ base $\Bee$ ( see \cite{dkz}). 

\begin{theorem}\label{pi-FNS-open}
	In every compact Hausdorff space with  the $\pi- \FNS$ property I Player has a winning strategy in open-open game.
\end{theorem}
\begin{pf}
	Assume that $X$ is a compact Hausdorff space there are $\pi-$ base $\Bee$ and $\s:\Bee\to[\Bee]^{<\omega}$ which witnessing the $\pi-\FNS$ property. We shall define a winning strategy $\sigma$ for I player in the game $G_\omega(X)$. 
Take any $U_0\in\Bee$ and put $\Cee_0=\sigma(\emptyset)=\{U_0\}$. Assume that we have just defined
 $$\sigma(\emptyset),\Dee_0,\ldots,\sigma(\Dee_0,\ldots,\Dee_{n-1}),\Dee_{n},$$ where   all families $\Dee_j$ and 
 $\sigma ( \Dee_0,\Dee_1 , \ldots, \Dee_j)$ are finite and  consists of non-empty open sets such for each 
 $U \in \sigma ( \Dee_0,\Dee_1 , \ldots, \Dee_j)$ there exists $V \in \Dee_{j+1}$ with $V \subseteq U .$ 
 Let $\Aaa_n=\{s(W):W\in \Dee_0\cup\Dee_1 \cup \ldots\cup \Dee_n\}.$ For each finite family $\Raa\subseteq\Aaa_n$ 
 with the non-empty intersection we choose $V_\Raa\in\Bee$ such that $V_\Raa\subseteq \bigcap\Raa.$ 
 Put $$\sigma(\Dee_0,\ldots,\Dee_{n})=\{V_\Raa\in\Bee:\Raa\in[\Aaa_n]^{<\omega}\text{ and } \bigcap\Raa\ne\emptyset\}.$$
  Suppose that $\bigcup_{n\in\omega}\bigcup\Dee_n$ is not dense in $X$. There is a set $V\in\Bee$ such that 
  $V\cap \bigcup_{n\in\omega}\bigcup\Dee_n=\emptyset$. For each $U\in\bigcup_{n\in\omega}\bigcup\Dee_n$ 
  there is $W\in s(V)\cap s(U)$ such that $W\cap U=\emptyset$ and $V\subseteq W.$ There is $n\in\omega$ such 
  that $$\Raa=\left\{W\in s(V):\exists(U\in \bigcup_{n\in\omega}\bigcup\Dee_n)\;W\cap U=\emptyset \text{ and } V\subseteq W\right\}
  \subseteq \Aaa_n.$$
  Hence there exists $V\in\Dee_{n+1}$ such that $V\subseteq \bigcap\Raa$, a contradiction with $\bigcap\Raa\cap\bigcup_{n\in\omega}\bigcup\Dee_n=\emptyset.$
\end{pf}

Similar theorem to the Theorem \ref{pog} is true for spaces with $\pi-\FNS$ property and skeletaly generated spaces.

\begin{theorem}\label{pi-FNS}
Every compact Hausdorff space with  the $\pi- \FNS$ property  is skeletally generated.
\end{theorem}

\begin{proof}
Assume that $X$ is compact Hausdorff space with  the $\pi- \FNS$ property. By Theorem \ref{pi-FNS-open} I Player has a winning strategy. Using  \cite[Theorem 12]{kp8} we get $X$ is skeletally generated.
\end{proof}

\section{Developable spaces have the FN property }
 Let $\mathcal{U}$ be a cover of a space $X$ and $x\in X$. Recall \cite{eng} that a set $\St(x,\mathcal{U})=\bigcup \{U\in \mathcal{U}:x\in U\}$ is called a \textit{star of the point $x$ with respect to the cover $\mathcal{U}$}. A sequence $\{\mathcal{W}_n\}_{n\in \omega}$ of covers of a space $X$ is called a \textit{development} whenever for every point $x\in X$ and an open set $U\subseteq X$ such that $x\in U$ there exists $i \in \omega$ such that $\St(x,\mathcal{W}_i) \subseteq U$. A space which has a development is called \textit{developable}.

A family $\mathcal{U}=\{U_t:t\in T\}$ of sets  is called a \textit{point-finite} if for every $x\in X$ the set $\{t\in T:x\in U_t\}$ is finite. We say that a cover $\mathcal{U}=\{U_t:t\in T\}$ is a \textit{refinement} of a cover $\mathcal{V}=\{V_s:s\in S\}$ if for every $t\in T$ there exists $s\in S$ such that $U_t\subseteq V_s$. 

\begin{proposition}
A topological space having a development consisting of point-finite covers has $\FN$ property.
\end{proposition}

\begin{proof}
Let $\{\mathcal{U}_n\}_{n\in \omega}$ be a a development  consisting of point-finite covers of  $X$. Without loss of generality it can be assumed that there exists a development $\{\mathcal{W}_n\}_{n\in \omega}$ consisting of point-finite covers such that the cover $\mathcal{W}_{n+1}$  is a refinement of the cover $\mathcal{W}_n$. 

For a family $\Aaa$ of sets we shall denote by $\Aaa^M$ subfamily of $\Aaa$ consisting of all maximal elements, i.e. of set $A\in\Aaa$ such that if $A\subseteq A'\in\Aaa$ then $A=A'$. Since each $\mathcal{W}_n$  is point-finite cover $\Wee^M_n$ is well defined and it is the point-finite cover.  Two distinct elements in $\mathcal{W}^M_n$ are incomparable by the inclusion. The sequence $\{\mathcal{W}_n^M\}_{n\in \omega}$ is the  development. Indeed, take a point $x\in X$ and a neighborhood $V$ of $x$.
Since $\{\mathcal{W}_n\}_{n\in \omega}$ is the development there is $n\in\omega$ such that $\St (x, \mathcal{W}_n)\subseteq V.$ If $x\in W\in\Wee^M_n\subseteq\Wee_n$ then $W\subseteq U.$ Therefore $\St (x, \mathcal{W}^M_n)\subseteq V.$
The family $\mathcal{B}=\bigcup \{\mathcal{W}^M_n : n\in \omega\}$ is a base for  the space $X$. 
For each $U \in \mathcal{B}$ put
$$\low(U)=\{U\} \text{ and } \up(U)=\{W\in \mathcal{W}^M_i: U\subseteq W, i\le \min\{n: U\in \mathcal{W}^M_n\}\}.$$

If $U\subseteq V$, then $ U\in \mathcal{W}^M_k$ and $V\in \mathcal{W}^M_n $ for $k<n$ or $n<k$. If $k<n$, then for $V$ there is $W\in \mathcal{W}^M_k$ such that $U\subseteq V\subseteq W$. Since elements of $\mathcal{W}^M_k$ are incomparable by the inclusion,  we have $U=V$. If $n<k$, then $V\in l(V)\cap u(U)$. 
\end{proof}
Since every metrizable space has a development consisting of point-finite covers \cite[Theorem 5.4.8]{eng}  we get the following
\begin{corollary}[\cite{bk}]
Metric spaces satisfy FN property.
\end{corollary}

\section{$\FNS$ property and absolutes}

We will need a notion of a small image and some facts connected with it.

Let $f$ be a map from a space $X$ to a space $Y$. A \textit{small image} of a set $U\subseteq X$ under a map $f$ is called a set $f^\#(U)=\{y\in Y: f^{-1}(y)\subseteq U\}$.


\begin{lemma}\label{frd}
Let $f$ be a map from $X$ to $Y$, then small image of a set $U\subseteq X $ under the map $f$ is given by $f^\#(U)=Y\setminus f(X\setminus U)$. 
\end{lemma}

\begin{proof}
We will show that $Y\setminus f(X\setminus U)=\{y\in Y: f^{-1}(y)\subseteq U\}$ for every set $U\subseteq X$.

Let $y\in Y\setminus f(X\setminus U)$, then $y\in Y\text{ and for every } x\in X \text{ such that } y=f(x) \text{ we have }x\notin X\setminus U$, hence $y\in Y \text{ and } f^{-1}(y)\subseteq U$.

Let $y\in Y$ and $f^{-1}(y)\subseteq U$, then $f^{-1}(y)\cap (X\setminus U)=\emptyset$. Hence $f\big(f^{-1}(y)\cap (X\setminus U)\big)=\{y\}\cap f(X\setminus U)=\emptyset$ and $\{y\}\subseteq Y\setminus f(X\setminus U)$.
\end{proof}

\begin{lemma}\label{f1}
Let $f\colon Z\to X$ be an irreducible map and $U,V\subseteq Z$ be open sets, then $U \cap V=\emptyset$ if and only if $f^\#(U) \cap f^\#(V)=\emptyset$. 
\end{lemma}

\begin{proof}
Let $U,V\subseteq Z$ and $U \cap V=\emptyset$. Suppose that $f^\#(U) \cap f^\#(V)\neq \emptyset$, then there exists $y\in Y$ such that $f^{-1}(y)\subseteq U\cap V$, this is a contradiction. 

Let  $f^\#(U) \cap f^\#(V)=\emptyset$. Suppose that $U \cap V\ne \emptyset$, then we get $\emptyset \ne X\setminus f\big( Z\setminus (U\cap V)\big)=f^\#(U\cap V)\subseteq f^\#(U) \cap f^\#(V)$ and this contradiction ends the proof.
\end{proof}

\begin{lemma}\label{f2}
Let $f$ be a map from $Z$ onto $X$. For every set $U\subseteq Z$ we have $f^\#(U) \subseteq f(U) $.
\end{lemma}

\begin{proof}
Let us take any set $U\subseteq Z$ and $y\in f^\#(U)$, then $f^{-1}(y)\subseteq U$. Hence there exists $x\in f^{-1}(y)\subseteq U$ and $y=f(x)\in f(U)$.
\end{proof}

We say that a continuous, closed map $f$ from $X$ onto $Y$ is \textit{irreducible} if $Y$ is not the image of any proper closed subset of  $X$. The \textit{absolute} of a  space $X$ is the space which is mapped and irreducibly onto $X$, and is such that any irreducible inverse image of the space $X$ is homeomorphic to it. We say that topological spaces are \textit{coabsolute} if their absolutes are homeomorphic.

\begin{lemma}\label{ptb}
If $\mathcal{B}_X$ is a $\pi$--base for $X$ and $f\colon Z\to X$ is an irreducible and closed map, then $\{f^{-1}(B): B\in \mathcal{B}_X\}$ is a $\pi$--base for $Z$.
\end{lemma}

\begin{proof}
We will show that for every open set $\emptyset\ne U\subseteq Z$ there exists a set $B\in \mathcal{B}_X$ such that $f^{-1}(B)\subseteq U.$ Let us take any nonempty, open set $U\subseteq Z$, then $Z\setminus U\ne Z$ and it is closed. Hence $X\setminus f(Z\setminus U)\ne \emptyset$ and it is open, then there exists an open set $B\in \mathcal{B}_X$ such that $B\subseteq X\setminus f(Z\setminus U)$. We get $\emptyset =B\cap f(Z\setminus U)=f(f^{-1}(B) \cap (Z\setminus U))$. Hence $f^{-1}(B) \cap (Z\setminus U)=\emptyset$, and $f^{-1}(B)\subseteq U$.
\end{proof}

\begin{theorem}\label{coabsolute}
Every space $Y$ which is co-absolute to a space $X$ with the $\pi-\FNS$ property has the $\pi-\FNS$ property.
\end{theorem}

\begin{proof}
Let a space $Z$ be an absolute of the space $X$ and $g\colon Z \to X$, $f\colon Z\to Y$ be irreducible, onto maps. Let $\mathcal{B}_X$ be a $\pi$-base for $X$ which has $\FNS$ property and $\s$ witnesses this property.

We will show that the family $\{f^\#g^{-1}(V):V\in \mathcal{B}_X\}$ is a $\pi$-base for $Y$. Let $U$ be an open set in $Y$. By Lemma \ref{ptb} there exists a set $V\in \mathcal{B}_X$ such that $g^{-1}(V)\subseteq f^{-1}(U)$. Hence by Lemma \ref{f2} $f^\#g^{-1}(V)\subseteq f^\#f^{-1}(U)\subseteq ff^{-1}(U)\subseteq U$.

We will show that the family $\{f^\#g^{-1}(V):V\in \mathcal{B}_X\}$ has $\FNS$ property. We define a map $\s_Z\colon\{f^\#g^{-1}(V):V\in \mathcal{B}_X\}\to  [\{f^\#g^{-1}(V):V\in \mathcal{B}_X\}]^{<\omega}$ by the formula
$$\s_Z(f^\#g^{-1}(U))=\{f^\#g^{-1}(W):W\in \s(U)\}.$$
Let $f^\#g^{-1}(V) \cap f^\#g^{-1}(U)=\emptyset$ for some $V, U\in \mathcal{B}_X$. By Lemma \ref{f1} it has to be $V\cap U=\emptyset$, then there exist disjoint sets $V',U'\in \s(U)\cap s(V)$. We have $ f^\#g^{-1}(V'),  f^\#g^{-1}(U')\in s(f^\#g^{-1}(V)) \cap \s(f^\#g^{-1}(U)) $, $f^\#g^{-1}(V)\subseteq f^\#g^{-1}(V')$ and $f^\#g^{-1}(U)\subseteq f^\#g^{-1}(U')$. By Lemma \ref{f1} we have $f^\#g^{-1}(V')\cap f^\#g^{-1}(U')=\emptyset$. This ends the proof.
\end{proof}
 Dugundji spaces were introduced in \cite{p}. Skeletally Dugundji spaces are skeletal analogue of Dugundji spaces, see \cite{kpv1}. 
We say that a compact space $X$ is Dugundji (skeletally Dugundji) whenever $X$ can be represented as the limit space of a continuous inverse system
$\displaystyle S=\{X_\alpha, p^{\beta}_\alpha, \alpha<\beta< w(X)\}$ such that $X_0$ is a compact metrizable  with open bonding (skeletal) maps and  each map $p^{\alpha+1}_\alpha$ has a metrizable kernel, see \cite{ha}. L. Shapiro \cite{sh} proved that every skeletally Dugundji spaces are co-absolute to a 0-dimensional Dugundji spaces. By Theorem 2.2.3 \cite{npba} ( see also Theorem 21 \cite{bk}) and Theorem \ref{coabsolute} every 0-dimensional skeletally Dugundji space has the $\pi-\FNS$ property.
  
\begin{question}
	Do openly generated spaces  have the $\FNS$ property?
\end{question}
\begin{question}
	Do skeletally generated spaces  have the $\pi-\FNS$ property?
\end{question}

\end{document}